\documentclass[final,preprint,12pt,authoryear]{elsarticle}
\usepackage{amsmath}
\usepackage{amssymb}
\usepackage{amsthm}
\usepackage{url}
\usepackage{hyperref}
\usepackage{natbib}
\usepackage{graphicx}
\usepackage{booktabs}
\usepackage{bm}
\let\top\intercal
\newtheorem{theorem}{Theorem}
\newtheorem{proposition}[theorem]{Proposition}
\newtheorem{lemma}[theorem]{Lemma}

\theoremstyle{definition}
\newtheorem{example}{Example}

\newtheorem{definition}[theorem]{Definition}

\newcommand{\reals}{\mathbb{R}}
\newcommand{\naturals}{\mathbb{N}}

\newcommand{\cT}{\mathcal{T}}
\newcommand{\cE}{\mathcal{E}}
\newcommand{\cH}{\mathcal{H}}
\newcommand{\cU}{\mathcal{U}}
\newcommand{\dens}[3]{{#1}^{#3}_{#2}}

\newcommand{\StudT}{\operatorname{T}}

\newcommand{\densrel}[3]{{#1}^{#3}_{#2}}
\newcommand{\ocset}{\cU}
\newcommand{\ocvar}{U}
\newcommand{\ocval}{u}
\newcommand{\rocvar}{Y}

\newcommand{\covvar}{Z}

\newcommand{\susset}{\cT}
\newcommand{\susvar}{T}
\newcommand{\susval}{t}

\newcommand{\nullpoint}{\cH_{\delta = \delta_0}}
\newcommand{\nulloneside}{\cH_{\delta \leq \delta_0}}

\DeclareMathOperator{\ex}{\mathbf E}
\DeclareMathOperator*{\exo}{\mathbf E}

\DeclareRobustCommand{\VANDER}[3]{#2}
\journal{SPL: E-values Special Issue}
\begin{document}

\begin{frontmatter}
% title page -------------------------------------------------------------------
\title{Supermartingales for One-Sided Tests: \\ Sufficient Monotone Likelihood Ratios are Sufficient}
\author[1,2]{Peter Gr\"unwald}
\author[1,3]{Wouter M. Koolen}
\affiliation[1]{organization={Centrum Wiskunde \& Informatica}, city={Amsterdam}, country={The Netherlands}}
\affiliation[2]{organization={Leiden University}, city={Leiden}, country={The Netherlands}}
\affiliation[3]{organization={Twente University}, city={Enschede}, country={The Netherlands}}

%\date{\today}

%\maketitle
\begin{abstract}
  The t-statistic is a widely-used scale-invariant statistic for testing the null hypothesis that the mean is zero. Martingale methods enable sequential testing with the t-statistic at every sample size, while controlling the probability of falsely rejecting the null. For one-sided sequential tests, which reject when the t-statistic is too positive, a natural question is whether they also control false rejection when the true mean is negative. We prove that this is the case using monotone likelihood ratios and sufficient statistics. We develop applications to the scale-invariant t-test, the location-invariant $\chi^2$-test and sequential linear regression with nuisance covariates.
\end{abstract}


%\item The sequential one-sided t-test of positive against zero effect size is in fact a supermartingale against any negative effect size.

%% Keywords
%\begin{keyword}
%  Sequential t-test \sep supermartingales \sep e-variables \sep group invariance \sep monotone likelihood ratio



%\end{keyword}

\end{frontmatter}

\section{The Problem}
We start by reviewing the setting of the one-sided sequential testing problem. Let
$\ocvar_1, \ocvar_2, \ldots$ be a stochastic process. Each $\ocvar_i$ is a random variable with support $\ocset_i \subset \reals$ and we abbreviate the sequence of outcomes by $\ocvar^n = (\ocvar_1, \ldots, \ocvar_n)$ and its domain by $\ocset^{(n)} = \ocset_1 \times \ldots \times \ocset_n$. The underlying filtration is $(\sigma(\ocvar^n))_{n \in \naturals}$.
Let $\{P_{\delta}: \delta \in \Delta\}$ be a 1-parameter statistical model for such a process, where the parameter domain $\Delta \subseteq \reals$ is an interval. We restrict attention to models in which we are allowed to condition just as in elementary probability. Specifically, we assume that, for all $\delta \in \Delta$, there exists a version of the conditional distribution $P_{\delta}(\ocvar_i = \cdot \mid \ocvar^{i-1})$ such that for all $i$,  for all $\ocval^{i-1} \in \ocset^{(i-1)}$, all $\delta \in \Delta$, the $\dens{P}{\delta}{\ocvar_i}\mid \ocval^{i-1} := P_{\delta}(\ocvar_i = \cdot \mid \ocvar^{i-1}= \ocval^{i-1})$ are mutually absolutely continuous, share the same support $\ocset_i$ and have densities $\dens{f}{\delta}{\ocvar_i}\mid \ocval^{i-1}$ relative to either Lebesgue or counting measure that coincide with their elementary conditional densities \cite[Section 9.6]{Williams91}. By these assumptions, we may fix any element $\delta_0 \in \Delta$ and represent  $P_{\delta}(\ocvar_i = \cdot \mid \ocvar^{i-1})$ by its density $\densrel{p}{\delta}{\ocvar_i}(\ocvar_i \mid \ocvar^{i-1})$  relative to $P_{\delta_0}(\ocvar_i \mid \ocvar^{i-1})$.  Here, following \cite{perez2024estatistics}, we denote the distribution and density of a measurable random quantity $V$ under $P_{\delta}$ by  $\dens{P}{\delta}{V}$ and  $\densrel{p}{\delta}{V}$, respectively. We shall reserve the symbol $p_{\delta}$ for the density of $P_{\delta}$ relative to $P_{\delta_0}$ --- these may also be thought of as likelihood ratios. When using densities relative to Lebesgue or counting measure we will use the symbol $f_{\delta}$ instead.

\paragraph{Simple Null}
We first consider testing a point null $\nullpoint = \{ P_{\delta_0}\}$ against an alternative $P_W$, where $W$ is an arbitrary prior measure on $\Delta$ and $P_W$ has density $\densrel{p}{W}{\ocvar^n}(\ocvar^n) := \int \densrel{p}{\delta}{\ocvar^n}(\ocvar^n) d W(\delta)$. We allow $W$ to be degenerate, i.e.\ put all its mass on a single point.
We may think of $(\densrel{p}{W}{\ocvar^n})_{n \in \naturals}$ as a \emph{likelihood ratio process}.
As is well known, this likelihood ratio process is a \emph{test martingale} (non-negative and starting from $1$) relative to null $\nullpoint = \{P_{\delta_0} \}$ \citep{ramdas2023savi}. Therefore, it handles optional stopping and anytime-validity. Specifically, under any stopping time $\tau$, the test which rejects $\nullpoint$ if $\densrel{p}{W}{\ocvar^{\tau}}(\ocvar^{\tau})\geq \alpha^{-1}$ has Type-I error bounded by $\alpha$.
The (super)-martingale property requires that for all $n$, all $\ocval^{n-1} \in \ocset^{(n-1)}$, the conditional one-step likelihood ratios ${\densrel{p}{W}{\ocvar_n}(\cdot \mid \ocval^{n-1})}$ are \emph{past-conditional e-variables}, i.e.
\begin{equation}\label{eq:testmart}
\ex_{\delta_0}
\left[{\densrel{p}{W}{\ocvar_n}(\ocvar_{n} \mid \ocval^{n-1})} \middle|  \ocval^{n-1} \right] \leq 1.
\end{equation}
This, in turn, readily follows from the standard cancellation argument:
\begin{equation}\label{eq:testmartb}
\ex_{\delta_0}\left[{\densrel{p}{W}{\ocvar_n}(\ocvar_{n} \mid \ocval^{n-1})}\middle| \ocval^{n-1} \right]= \int_\ocval \frac{d \dens{P}{W}{\ocvar_n }|\ocval^{n-1}}{d \dens{P}{\delta_0}{\ocvar_n}|\ocval^{n-1}}(\ocval) d \left( \dens{P}{\delta_0}{\ocvar_n}| \ocval^{n-1} \right)  = 1.
\end{equation}

\paragraph{One-Sided Tests}
Now fix $\delta^+ > \delta_0$ and consider the one-sided null hypothesis
\[\nulloneside := \{P_{\delta}: \delta \in \Delta, \delta \leq \delta_0\}.\]
As alternative we may take either the Bayesian point alternative $\cH_1 = \{P_W\}$ with $W$ a prior on $\{\delta \in \Delta: \delta \geq \delta^+ \}$ or a composite alternative $\cH_1 \subseteq
\{P_{\delta}: \delta \in \Delta, \delta \ge \delta^+\}$.  For simplicity we concentrate on the case with $\cH_1 = \{P_{\delta^+}\}$ for now, returning to the general case in Section~\ref{sec:discussion}.

The question we aim to answer is: \emph{does the likelihood ratio process $(\densrel{p}{\delta^+}{\ocvar^n})_{n \in \naturals}$ still constitute a test supermartingale for the enlarged null hypothesis $\nulloneside$?} In other words, does~\eqref{eq:testmart} still hold with the expectation taken over $P_{\delta}$ instead of $P_{\delta_0}$ (while the $\delta_0$ hidden by the density notation $p_W$ remains fixed)?

In applications, the parameter $\delta$ often represent a notion of effect size and one would then expect that, at least under some further conditions,~\eqref{eq:testmart} still holds in this case  --- indeed, for some special cases such as the test supermartingales appearing in \cite{TurnerLG24} ($k\times 2$ tables) and \cite{TerschurePLG24} (logrank test), the test supermartingale property~\eqref{eq:testmart} was shown to hold. \emph{But how general is this phenomenon?}
Below we first indicate why existing results do not directly tell us either way; Our main result, Lemma~\ref{lem:main}, provides a condition under which the supermartingale property does hold. To state it and earlier results, we need the following property.
\begin{definition}[Monotone Likelihood Ratio Property]\label{def:mlr}
Let $\susvar$ be a random variable with common support $\susset$ under all $P_{\delta}$ with $\delta \in \Delta$. 
We say that $\susvar$ satisfies the \emph{Monotone Likelihood Ratio (MLR) Property} if
for all $\delta_0, \delta^+ \in \Delta$ with $\delta_0 \le \delta^+$, the likelihood ratio $\densrel{p}{\delta^+}{\susvar}(\susval)$ is increasing\footnote{throughout, increasing will be used in the non-strict sense} as a function of $\susval \in \susset$.
\end{definition}
\noindent
\cite[Proposition 3]{GrunwaldHK24} implies the following:
\begin{proposition}\label{prop:monotoneb}
 Suppose that $\susvar$ satisfies the MLR Property and fix $\delta \le \delta_0 \le \delta^+$ with  $\delta, \delta_0,\delta^+ \in \Delta$.
 Then $\ex_\delta[\densrel{p}{\delta^+}{\susvar}(\susvar)] \leq 1$, i.e.\ the likelihood ratio $\densrel{p}{\delta^+}{\susvar}(\susvar)$ is an e-variable for $\nulloneside$.
\end{proposition}

\noindent
\ref{appx:proofs} includes a direct proof.
This proposition can and has been used to show that some likelihood ratios $\densrel{p}{\delta^+}{\susvar_n}(\susvar_n)$, at fixed sample size $n$, of statistic $\susvar_n = \susval_n(\ocvar^n)$ set to some fixed function of $\ocvar^n$, provide e-\emph{variables}. Yet, it tells us nothing about whether the sequence $(\densrel{p}{\delta^+}{\susvar_n}(\susvar_n))_{n \in \naturals}$ is a supermartingale.
\begin{example}[t-test]\label{ex:ttest}
Let us illustrate this using the standard anytime-valid t-test as presented by both \cite{GrunwaldHK24} and \cite{perez2024estatistics}. Here $\Delta= \reals$ and $P_{\delta}^{Y^\infty}$ expresses that the data $\rocvar_1, \rocvar_2, \ldots$ are i.i.d.\ $\sim N(\mu,\sigma)$ for some  $(\mu,\sigma)$ such that  $\delta = \mu/\sigma$. Thus $\delta$ is not sufficient to determine the distribution of the data $\rocvar_1, \rocvar_2\ldots$ Yet the sequence of \emph{maximal invariants} \citep{perez2024estatistics} $\ocvar_1, \ocvar_2, \ldots$ with
$
\ocvar_i := \frac{\rocvar_i}{|\rocvar_1|}, 
$
has the same distribution $P_{\delta}^{U^\infty}$ under all $N(\mu,\sigma)$ with given effect size $\mu/\sigma = \delta$. It is invariant to rescaling, e.g.\ changing the unit of measurement of the $Y_i$ (note that the support $\ocset_1 = \{\pm 1\}$ is discrete whereas $\ocset_i = \reals$ for $i > 1$). The standard anytime-valid t-test is defined in terms of the likelihood ratio process $(\densrel{p}{\delta^+}{\ocvar^n})_{n \in \naturals}$ with $P_{\delta}$ as above. By the argument~\eqref{eq:testmartb} above, it provides a test-martingale in the filtration $(\sigma(\ocvar^{n}))_{n \in \naturals}$ relative to the null $\nullpoint$.
Indeed, as 
\cite[Section 4.5]{wang2024anytimev} write for the case $\delta_0=0$ (notation adapted):
\begin{quote}
.... \cite[Corollary 8]{perez2024estatistics} show that at any fixed sample size $n$ the scale-
invariant t-likelihood ratio $\dens{p}{\delta^+}{\ocvar^n}$ is actually an e-value for the larger one-sided null $\cH_{\delta \leq 0}$ if $\delta^+ \geq 0$ $[\ldots]$
meaning that one can non-sequentially test that null using this statistic. However,
it is unclear if this t-likelihood ratio process  is an e-process for $\cH_{\delta \leq 0}$. \emph{We leave this question open for future investigation}.
\end{quote}
...an observation which prompted nervousness among some members of the e-community, who had simply assumed the t-likelihood ratio to be a super-martingale for $\nulloneside$, and now discovered that it was not straightforward to prove.
\end{example}
\paragraph{The Difficulty}
At first sight, the problem might seem easy to solve: if we could show that for every sample size $n$ and past $\ocval^{n-1} \in \ocset^{(n-1)}$,
the \emph{conditional likelihood ratio}
\begin{align}\label{eq:condlr}
    \densrel{p}{\delta^+}{\ocvar_n}(\ocval \mid \ocval^{n-1})
\end{align}
is monotone in the next outcome $\ocval$, then we could still apply Proposition~\ref{prop:monotoneb} above pointwise, for each $n$ and conditional on each $\ocval^{n-1}$, and we would be done. This is how \cite{TurnerLG24,TerschurePLG24} proceed: in their settings, the $\ocvar_i$ are i.i.d.\ making it easy to employ Proposition~\ref{prop:monotoneb} pointwise.

The t-test setting is more complicated, as the maximal invariants $\ocvar_i$ are not i.i.d.  One might be tempted to exploit that the original, uncoarsened data $\rocvar_i$ are i.i.d., but unfortunately the t-likelihood ratio simply is not an e-process relative to the corresponding filtration $(\sigma(\rocvar^n))_{n \in \naturals}$, leading to trouble when optional stopping is desired.
This subtle point is discussed in detail by \cite{perez2024estatistics}. It reflects the more general fact that in group-invariant testing, of which the t-test is merely a very special case,  \emph{even when the raw data are i.i.d.\ we often need to deal with $\ocvar_i$ that are not i.i.d.}. This is further illustrated by all examples in \ref{sec:examples}.
One might then directly set to prove that the conditional likelihood ratio~\eqref{eq:condlr} is monotone in  $\ocvar_n$. However,
\cite{wangpersonal} showed via numerical experiments that for the t-test it is \emph{not}.

So what to do? In Theorem~\ref{thm:main} below we show the following: if for all $n$, there exists a \emph{sufficient statistic} $\susvar_n= \susval_n(\ocvar^n)$ for the model $\dens{P}{\delta}{\ocvar^n}$ such that the MLR Property  holds for $\susvar_n$, then $(\densrel{p}{\delta^+}{\ocvar^n})_{n \in \naturals}$ is a test supermartingale relative to $\nulloneside$ after all.
We will see that within the t-test setting, the  t-statistic provides just such a sufficient statistic. Thus, our result resolves the issue for the t-test likelihood ratio but also, as we will show in \ref{sec:examples}, for several other cases of interest.
\section{The Solution}
Let $\susvar_n$ be a sufficient statistic for model $\{P_{\delta}: \delta \in \Delta \}$ and data $\ocvar^n$. Recall that this means that, with $\dens{f}{\delta}{\susvar_n}$ and  $\dens{f}{\delta}{\ocvar^n}$
the densities of $\dens{P}{\delta}{\susvar_n}$ and  $\dens{P}{\delta}{\ocvar^n}$ relative to some background measures $\nu^{\susvar_n}$ and $\nu^{\ocvar^n}$,  there exist functions $\susval_n: \ocset^{(n)} \rightarrow \reals$ and $q_n: \ocset^{(n)} \rightarrow \reals^+$ such that  $\susvar_n = \susval_n(\ocvar^n)$ and for all $\delta \in \Delta$, all $\ocval^n \in \ocset^{(n)}$, we have
\begin{equation}\label{eq:sufficiency}
\dens{f}{\delta}{\ocvar^n}(\ocval^n)~=~\dens{f}{\delta}{\susvar_n}(\susval_n(\ocval^n) ) \cdot q_n(\ocval^n).
\end{equation}
%In particular, the likelihood ratio simplifies to $\densrel{p}{\delta}{\ocvar^n}(u^n) = \densrel{p}{\delta}{\susvar_n}(t_n(u^n))$, and the conditional likelihood ratio simplifies to $\densrel{p}{\delta}{\ocvar_n}(u_n|u^{n-1})
% = \frac{
%   \densrel{p}{\delta}{\susvar_n}(t_n(u^n))
% }{
%   \densrel{p}{\delta}{\susvar_{n-1}}(t_{n-1}(u^{n-1}))
% }$.
%
With that, we are ready to state our main results, both proved further below:
\begin{lemma}\label{lem:main} Let, for each $n$, $\susvar_n$ be a sufficient statistic  such that the MLR Property (Def.~\ref{def:mlr}) holds with $\susvar$ set to $\susvar_n$.
Then for any $\delta_0, \delta^+ \in \Delta$ with $\delta_0 < \delta^+$, and all $n\geq 1$, for all $\ocval^n \in \ocset^{(n)}$,  we have
\[\densrel{p}{\delta^+}{\ocvar_n}(\ocval_n |\ocval^{n-1})~=~
\densrel{p}{\delta^+}{\susvar_n}(\susval_n(\ocval^n) |\ocval^{n-1}).
\]
and the conditional likelihood ratio
$\densrel{p}{\delta^+}{\susvar_n}(\susval |\ocval^{n-1})$
is increasing in $\susval$.
%\text{\ where\ } \dens{f}{\delta^+}{\ocvar_n}(\ocval |\ocval^{n-1}) := \frac{\dens{p}{\delta^+}{\ocvar_n}(\ocval \mid \ocval^{n-1})}{\dens{p}{\delta_0}{\ocvar_n}(\ocval\mid \ocval^{n-1})}.
\end{lemma}

%WMK: check why this is equal and not just proportional to. Oh, because both are likelihood \emph{ratios} and not likelihoods. Duh.

Importantly, the likelihood ratio $
{\densrel{p}{\delta^+}{\ocvar_n}(\ocval \mid \ocval^{n-1})}
$ need \emph{not} be monotone in $\ocval$ and indeed it is not in the t-test setting; this caused the perceived difficulty of the problem. Combining the lemma with Proposition~\ref{prop:monotoneb} gives:
\begin{theorem}\label{thm:main}
Let $(\susvar_n)_{n \in \naturals}$ be a sequence of sufficient statistics satisfying the monotone likelihood ratio property (Def.~\ref{def:mlr}). Then the process $(\prod_{i=1}^n \densrel{p}{\delta}{\susvar_i}(T_i \mid \ocvar^{i-1}))_{n \in \naturals}$ is identical to the likelihood ratio process $(\densrel{p}{\delta}{\ocvar^n})_{n \in \naturals}$ and both are test supermartingales relative to $\nulloneside$.
\end{theorem}

\begin{example}[t-test, Continued]\label{ex:t2}
  Here are three well-known facts: (a) if data $Y_1,Y_2,\ldots$ are sampled i.i.d.\ from a normal $N(\mu,\sigma)$ with effect size $\delta = \mu/\sigma$, then the t-statistic
  \begin{equation}\label{eq:t.stat}
    \susvar_n
    ~=~
    t_n(U^n)
    ~=~ \frac{
      \frac{1}{\sqrt{n}} \sum_{i=1}^n U_i
    }{
      \sqrt{\frac{1}{n-1}\left(\sum_{i=1}^n U_i^2 - \frac{1}{n} (\sum_{i=1}^n U_i)^2\right)}
    }
    \qquad
    \text{with}
    \qquad
    U_i ~=~ \frac{Y_i}{|Y_1|}
    .
  \end{equation}
  at sample size $n$ has a noncentral t-distribution with $\nu := n-1$ degrees of freedom and noncentrality parameter $\lambda:= \sqrt{n} \delta$.
Let us denote the (Lebesgue) density of this distribution by $\dens{f}{\StudT(\nu, \lambda)}{}$. (b) Also, as  easily established, for each $n$, the t-statistic $\susvar_n$ is sufficient for the coarsening $(\ocvar_1, \ldots, \ocvar_n)$, see also \citep[Example 1]{perez2024estatistics}. That is, for each $n$, as an instance of~\eqref{eq:sufficiency}, we can write (with $\lambda_0 = \sqrt{n} \delta_0$ and $\lambda^+ = \sqrt{n} \delta^+$):
\begin{equation}\label{eq:t-test}
\densrel{p}{\delta^+}{\ocvar^n}(\ocval^n) ~=~
\frac{\dens{f}{\StudT(\nu, \lambda^+)}{}(\susval_n(\ocval^n))}{
\dens{f}{\StudT(\nu, \lambda_0)}{}(\susval_n(\ocval^n))}.
\end{equation}
Finally, (c) for fixed $\nu$, the family of noncentral t-distributions has the monotone likelihood ratio property \citep{Kruskal1954TheFunctions}. That is, the ratio
${\dens{f}{\nu, \lambda^+}{}(\susval)}/{
\dens{f}{\nu, \lambda_0}{}(\susval)}$ is increasing in $\susval$ when $\delta^+ \ge \delta_0$.
Taken together, facts (a)--(c) imply via Lemma~\ref{lem:main} and Theorem~\ref{thm:main} that the
likelihood ratio process $(\densrel{p}{\delta^+}{\ocvar^n}(\ocvar^n))_{n \in \naturals}$ is a test supermartingale relative to the one-sided composite null $\nulloneside$, answering \cite{wang2024anytimev}'s question in the affirmative.
\end{example}

A collection of examples, including the $\chi^2$-test, linear regression and label agnostic Bernoulli are presented in \ref{sec:examples}. We now proceed with the proofs of Lemma~\ref{lem:main} and Theorem~\ref{thm:main}. The proofs are embarrassingly simple once one realizes that to establish the supermartingale property one should consider the conditional density of $\susvar_n$ rather than $\ocvar_n$ --- \emph{that} realization is the real contribution of this note.
\begin{proof}[Proof of Lemma~\ref{lem:main}]
  Let's fix $u^{n-1} \in \ocset^{(n-1)}$ throughout the argument. By sufficiency~\eqref{eq:sufficiency}, the likelihood ratio of $u_n \in \ocset_n$ conditioned on $u^{n-1}$ is
  \begin{equation}\label{eq:suffcons}
    p_{\delta_+}^{U^n}(u_n|u^{n-1})
    ~=~
    \frac{
      p_{\delta_+}^{T_n}(t_n(u^n))
    }{
      p_{\delta_+}^{U^{n-1}}(u^{n-1})
    }
    ,
  \end{equation}
  and given $u^{n-1}$ this is a function of $u_n$ only through $t_n(u^n)$. For any possible $t \in \{t_n(u^{n-1}, u_n) \mid u_n \in \ocset_n\}$ the conditional likelihood ratio, using \eqref{eq:suffcons}, equals
  \[
    p_{\delta_+}^{T_n}(t|u^{n-1})
    ~=~
    \ex_{\delta_0}\left[
      p_{\delta_+}^{U_n}(U_n|u^{n-1})
      \middle|
      u^{n-1}, t_n(u^{n-1}, U_n) = t
    \right]
    ~=~
    \frac{
      p_{\delta_+}^{T_n}(t)
    }{
      p_{\delta_+}^{U^{n-1}}(u^{n-1})
    }
    .
  \]
  This is increasing in $\susval$ by assumption, showing the second claim. The first claim follows from one more application of \eqref{eq:suffcons}, i.e.
  \[
    p_{\delta_+}^{T_n}(t_n(u^n)|u^{n-1})
    ~=~
    \frac{
      p_{\delta_+}^{T_n}(t_n(u^n))
    }{
      p_{\delta_+}^{U^{n-1}}(u^{n-1})
    }
    ~=~
    p_{\delta_+}^{U^n}(u_n | u^{n-1})
    .
    \qedhere
  \]
\end{proof}

\begin{proof}[Proof  of Theorem~\ref{thm:main}]
We apply Proposition~\ref{prop:monotoneb} for each $n$ and each $\ocval^{n-1} \in \ocset^{(n-1)}$, with $\susvar$ replaced by $\susvar_n$ and $\densrel{p}{\delta^+}{\susvar}$ replaced by
$\densrel{p}{\delta^+}{\susvar_n} \mid \ocval^{n-1}$.
By Lemma~\ref{lem:main}, the MLR property holds for these \emph{conditional} densities, so the proposition can be applied and gives that for each $n$ and $\ocval^{n-1}$, it holds that $\densrel{p}{\delta+}{\susvar_n} \mid \ocval^{n-1}$ is an e-variable on domain $\ocset^{(n)}$ conditional on past data $\ocvar^{n-1}=\ocval^{n-1}$, i.e.\ for all $\delta \leq \delta_0$, we have:
\begin{equation}
    \ex_\delta\left[\densrel{p}{\delta^+}{\susvar_n}(\susvar_n \mid \ocval^{n-1}) \middle|  \ocval^{n-1}\right] \leq 1
    \end{equation}
and hence the products  $(\prod_{i=1}^n \densrel{p}{\delta}{\susvar_i}(\cdot \mid \ocvar^{i-1}))_{n \in \naturals}$ of these past-conditional e-variables constitute a supermartingale (see Proposition 2 of \cite{GrunwaldHK24}). But from the first claim of Lemma~\ref{lem:main} we know that for all $n$ and $\ocval^{n-1} \in \ocset^{(n-1)}$, we have for each factor in the product:
$\densrel{p}{\delta+}{\susvar_n}(\susvar_n \mid \ocval^{n-1})=
\densrel{p}{\delta+}{\ocvar_n}(\ocvar_n \mid \ocval^{n-1})$.
Since also for each $n$, it holds that $\densrel{p}{\delta}{\ocvar^n}(\ocvar^n)= \prod_{i=1}^n \densrel{p}{\delta}{\ocvar_i}(\ocvar_i \mid \ocvar^{i-1})$, the processes  $(\prod_{i=1}^n \densrel{p}{\delta}{\susvar_i}(\cdot \mid \ocvar^{i-1}))_{n \in \naturals}$  and  $(\densrel{p}{\delta}{\ocvar^n})_{n \in \naturals}$ must coincide.
\end{proof}

\section{Further Discussion and Future Work}
\label{sec:discussion}
\paragraph{Priors on Alternative and Growth-Rate Optimality Considerations}
What if we put a general prior $W$ on $\{\delta \in \Delta: \delta \geq \delta_0\}$ instead of the point-mass on $\delta^+ \ge \delta_0$ that we implicitly used in Lemma~\ref{lem:main} and Theorem~\ref{thm:main} above? Following exactly the same steps as above we see that both results still hold with $\densrel{p}{\delta^+}{\cdot}$ replaced everywhere by $\densrel{p}{W}{\cdot}$ for an arbitrary such prior $W$. What can we say about the quality of the general supermartingale $(\densrel{P}{W}{\ocvar_n})_{n \in \naturals}$, measured in the standard sense of \emph{e-power} or, equivalently, \emph{log-optimality} \citep{ramdas2023savi}?
First, note that for every stopping time $\tau$, we have that $\densrel{p}{W}{\ocvar^{\tau}}$ is an e-variable. Following standard definitions, we call this e-variable GRO (growth-rate optimal) or equivalently \emph{log-optimal} or \emph{num\'eraire} \citep{larsson2024numeraire} for alternative $\{P_W\}$ on sample space $\ocset^{(\tau)}$ against null $\cH_0$ if it achieves
\begin{equation}\label{eq:GRO}
\max_{S \in \cE(\cH_0; \ocvar^{\tau})} \ex_{P_W}\left[ \ln S(\ocvar^{\tau}) \right]
\end{equation}
where $\cE(\cH_0; \ocvar^{\tau})$ is the set of all e-variables  relative to  null hypothesis $\cH_0$,  that can be written as a function of $\ocvar^{\tau}$. We first note that the likelihood ratio $\densrel{p}{W}{\ocvar^{\tau}}$ is clearly GRO (growth-rate-optimal) for $\{P_{W} \}$ against null $\cH_0 := \nullpoint = \{P_{\delta_0}\}$,  i.e.\ it achieves~\eqref{eq:GRO} with
$\cH_0$  replaced by $\nullpoint$. This follows from the general fact that likelihood ratios are growth-rate optimal for simple nulls \citep{GrunwaldHK24,ramdas2023savi}. But it turns out that 
 $\densrel{p}{W}{\ocvar^{\tau}}$ is also GRO (growth-rate-optimal) for $\{P_{W} \}$ against one-sided null  $\cH_0 := \nulloneside$. To see this, note that, for any e-variable $S$, if $S \in \cE(\cH_a; \ocvar^{\tau})$ and $S \in \cE(\cH_b; \ocvar^{\tau})$ with $\cH_a \subset \cH_b$, and $S$ is growth-optimal relative to $\cH_a$, then it must also be growth-optimal relative to $\cH_b$ with the same alternative (the maximum in~\eqref{eq:GRO} is taken over a smaller set if we set $\cH_0 = \cH_b$ rather than $\cH_0 = \cH_a$). Since this holds for \emph{every} stopping time $\tau$, we can in fact infer the much stronger conclusion that the process $(\densrel{p}{W}{\ocvar^n})_{n \in \naturals}$  has the LOAVEV (log-optimal anytime-valid e-value) property for alternative $P_{W}$ against null  $\nulloneside$, meaning that it is growth-optimal in the strongest possible sense \citep{KoolenG21}.
In a similar manner one can show that Theorem~\ref{thm:main} also implies that, for any stopping time $\tau$, we have that 
$\densrel{p}{\delta^+}{\ocvar^{\tau}}$ is
GROW (growth-rate optimal in the worst case) for the one-sided composite alternative $\cH_{\delta \geq \delta^+}$ against $\nulloneside$, i.e.\ it achieves~\eqref{eq:GRO} with $\ex_{\ocvar^{\tau}\sim P_W}$ replaced by $\inf_{\delta \geq \delta^+} \ex_{\ocvar^{\tau}\sim P_{\delta}}$.

\paragraph{Future Work}
In several  classical testing scenarios that are closely related to the ones we present here, the minimal sufficient statistic is multivariate --- a case in point is the setting of Hotelling's $T^2$ test. It would be interesting (but certainly challenging) to see if our arguments can be extended to such cases. 
\section{Acknowledgements}
We thank Aaditya Ramdas and Hongjiang Wang for noticing the issue that was adressed in this paper. Peter Gr\"unwald has been supported by ERC ADG project No 101142168 (FLEX).

\DeclareRobustCommand{\VANDER}[3]{#3}
\bibliographystyle{abbrvnat}
\bibliography{SAVI,master,references}

\newpage
\appendix

\section{Proof for completeness}\label{appx:proofs}
\begin{proof}[Proof of Proposition~\ref{prop:monotoneb}]
It is a well-known fact that the monotone likelihood property implies stochastic dominance \citep{lehmann1986testing}; see \citep{brown1981variation} for general background. That is, for any increasing function $g: \susset \rightarrow \reals$ we must have that
\[\ex_\delta[g(\susvar)]\] is increasing in $\delta\in \Delta$. Applying this with $g(\susval) = \densrel{p}{\delta^+}{\susvar}(\susval)$ (which,  by the assumed MLR property, must itself be increasing), we get
that $h(\delta) := \ex_\delta[\densrel{p}{\delta^+}{\susvar}(\susvar)]$ is increasing in $\delta$. But  we must have $h(\delta_0)=1$ by the standard cancellation argument (as in~\eqref{eq:testmartb}, but without conditioning). It follows that $h(\delta) \leq 1$ if $\delta \leq \delta_0$.
\end{proof}

\section{Further Examples}\label{sec:examples}

\subsection{Variance of interest, translation nuisance ($\chi^2$-test)}
Let us assume that data $Y_1,Y_2,\ldots$ are i.i.d.\ normal $N(\mu, \sigma^2)$ with mean $\mu \in \reals$ and variance $\sigma^2 > 0$. Let us test the null hypothesis that the variance is $\sigma_0$ while the mean is arbitrary, against the alternative that the variance is $\sigma_+ > \sigma_0$ while the mean is arbitrary and treated as nuisance. The mean can be eliminated by coarsening the data to $U_i = Y_i - Y_1$. Then the likelihood ratio is
\begin{equation}\label{eq:chi-test}
  p^{U^n}_{\sigma_+}(u^n)
  ~=~
  \left(\frac{\sigma_0}{\sigma_+}\right)^{n-1} \exp \left(\frac{1}{2} \left(\frac{1}{\sigma_0^2} - \frac{1}{\sigma_+^2}\right) Q_n\right)
\end{equation}
where $Q_n := \sum_{i=1}^n U_i^2 - \frac{1}{n} \left(\sum_{i=1}^n U_i\right)^2$ is the $\chi^2$ statistic, i.e.\ the unnormalised empirical variance among the $U^n$ (or, equivalently, the $Y^n$). We see that $Q_n$ is a sufficient statistic. We also see that the likelihood ratio is increasing in $Q_n$ whenever $\sigma_+ \ge \sigma_0$, thus establishing the MLR property. We may also observe that $\frac{Q_n}{\sigma^2}$ has a (central) chi squared distribution, and write the above as
\[
  p^{U^n}_{\sigma_+}(u^n)
  ~=~
  \frac{
    f_{\chi^2(n-1, \sigma_+^2)}\left(Q_n\right)
  }{
    f_{\chi^2(n-1, \sigma_0^2)} \left(Q_n\right)
  },
\]
where $f_{\chi^2(\nu, s)}$ is the $\chi^2$ density with $\nu$ degrees of freedom and scaling $s>0$. By Theorem~\ref{thm:main}, we conclude that~\eqref{eq:chi-test} is a supermartingale under any generating variance $\sigma^2 \in [0, \sigma_0^2]$. By invariance of \eqref{eq:chi-test} under negating $\sigma_0^2$, $\sigma_+^2$ and $Q_n$ together, we see that when $\sigma_0 > \sigma_+$, \eqref{eq:chi-test} is a supermartingale under any generating variance $\sigma > \sigma_0$.

\subsection{Linear Regression}
We aim to test whether one of the coefficients of a linear
regression model is zero under Gaussian error assumptions.
The observations are of the form $(X_1, \rocvar_1, \covvar_1), \dots, (X_n, \rocvar_n, \covvar_n)$, with scalar
$X_i,\rocvar_i \in \reals$ and vector $\covvar_i\in \reals^{d}$ for each $i$. We consider the linear model given by
\begin{equation*}
  \rocvar_i ~=~ \delta \sigma X_i + \beta^{\top} \covvar_i + \sigma\varepsilon_i,
\end{equation*}
where $\delta \in \reals$, $\beta\in \reals^d$ and $\sigma\in \reals^+$ are the
parameters, and $\varepsilon_1, \dots,\varepsilon_n$ are i.i.d.\ errors with
standard Gaussian distribution $N(0,1)$. The parameter of interest is the effect size $\delta$, while we treat the coefficients $\beta$ and scale $\sigma$ both as nuisance. As such, this linear regression example is a generalisation of the t-test of Examples~\ref{ex:ttest} and~\ref{ex:t2}, which is recovered for $d=0$ (so that $\beta$ and the $Z_i$ trivialise away) and $X_i = 1$.
Based on likelihoods calculated by \cite{BhowmikK07}, \citet[Example 4.2]{perez2024estatistics}
develop a test martingale, which we reproduce in \eqref{eq:linreg} below, in the form of a likelihood ratio for the following testing problem:
\begin{equation}\label{eq:lin_reg_test}
  \cH_0: \delta = \delta_0
  \qquad
  \text{vs}
  \qquad
  \cH_1: \delta = \delta_+.
\end{equation}
A test for ``no effect'' is obtained by taking $\delta_0 = 0$; see also the very closely related treatment by \cite{lindon2024anytimevalidlinearmodelsregression}.
We revisit the construction of the process, with the goal of finding a sufficient statistic with the MLR property, allowing us to show that the resulting likelihood ratio remains a test supermartingale against the larger null hypothesis $\nulloneside$.

We present the development from the perspective of coarsening. To this end, define the vectors $\rocvar^n = (\rocvar_1, \dots, \rocvar_n)^{\top}$ and
$X^n = (X_1, \dots, X_n)^{\top}$, and the $n\times d$ matrix
$\bm{\covvar}_n=[\covvar_1,\dots,\covvar_n]^{\top}$ whose rows are the nuisance covariate vectors $\covvar_1, \dots,\covvar_n$. Let $r_n$ denote the rank of $\bm{\covvar}_n$ (note that $r_n$ increases with $n$, and it typical equals $r_n = \min\{n,d\}$, but may grow more slowly when some $Z_i$ are linearly dependent.) Let us also abbreviate $k := n - r_n$, on which we suppress the dependence on $n$ to reduce clutter. Our filtration will be defined as $(\sigma(M^n))_{n \in \naturals}$ where $M^n = (M_1,\ldots, M_n)$ and
$M_n=(
\frac{\bm{A}_n \rocvar^n}{\|\bm{A}_n \rocvar^n\|},
  X^n, \bm{\covvar}_n)$,
  where $\bm{A}_n$ is an $k \times n$  matrix whose
columns form an orthonormal basis for the orthogonal complement of the column
space of $\bm{\covvar}_n$~\citep{BhowmikK07}.
It follows that $\bm{A}_n \bm{A}_n^\top=\bm{I}_k$ and $\bm{A}_n^\top \bm{A}_n = \bm{I}_n - \bm \covvar_n(\bm \covvar_n^\top\bm \covvar_n)^\dagger\bm \covvar_n^\top$ where we can take ${}^\dagger$ to be any pseudo-inverse, and $\bm{I}_n$ is the $n\times n$ identity matrix.
In order to compute the likelihood ratio for $M_n$, we assume that the mechanism that generates $X^n$ and $\bm{\covvar}_n$ is the same under both
hypotheses, so that we only need to consider the distribution of ${\ocvar}_n = \frac{\bm{A}_n{\rocvar}^n}{\|\bm{A}_n {\rocvar}^n\|} \in S^{k-1} \subseteq \reals^k$ conditionally on $X^n$ and $\bm{\covvar}_n$. For interpretation, observe that $\bm A_n^\top \bm A_n Y^n$ is the residual of least squares regression of the labels $Y^n$ onto the covariates $\bm Z_n$. Then $\bm A_n^\top U_n$ is that residual normalised to unit length. And finally $U_n$ is a compressed representation of that residual where all zero directions are dropped. By construction, the coarsening $U_n$ has distribution solely dependent on $\delta$, and is independent of the nuisance parameters $\beta$ and $\sigma$. Let us characterise that distribution according to $P_{\delta, \beta, \sigma}$. First, before the normalisation to unit length, we have
\[
  \bm A_n Y^n
  ~\sim~
  N\left(
    \delta \sigma b_n,
    \sigma^2 \bm I_k
  \right)
  \qquad
  \text{where}
  \qquad
  b_n ~:=~ \bm{A}_n X^n \in \reals^k
  .
\]
We may further decompose $\bm A_n Y^n$ as a component along $b_n$ and a component orthogonal to $b_n$. To that end, let $\bm P_n = \bm I_k
    -
    \frac{
      b_n b_n^\top
    }{
      b_n^\top b_n
    }
  $ denote the projection on the orthogonal complement of $b_n$. We find that the inner product and orthogonal complement are independent, with
  \[
    \frac{b_n^\top}{\|b_n\|} \frac{\bm A_n Y^n}{\sigma}
    ~\sim~
    N\left(\delta \|b_n\|, 1\right)
    \qquad
    \text{and}
    \qquad
    \bm P_n \frac{\bm A_n Y^n}{\sigma}
    ~\sim~
    % N\left(
    %   \delta \left(\bm I_k
    %   -
    %   \frac{
    %     b_n b_n^\top
    %   }{
    %     b_n^\top b_n
    %   }
    %   \right) b_n,
    %   \left(\bm I_k
    %   -
    %   \frac{
    %     b_n b_n^\top
    %   }{
    %     b_n^\top b_n
    %   }
    %   \right)^2
    % \right)
    % ~=~
    N\left(
      0,
      \bm P_n
    \right)
    .
  \]
  The dependence on the parameter $\delta$ is fully localised to the former. Moreover, the square norm of that latter orthogonal complement satisfies
  \[
    \left\|
      \bm P_n \frac{\bm A_n Y^n}{\sigma}
    \right\|^2
    ~\sim~
    \chi^2_{k-1}
    .
  \]
  Finally, the properly scaled ratio of a unit-variance normal and an independent $\chi^2$ random variable has a non-central Student-$t$ distribution. In our case, the scaled ratio (in which the $\sigma$ dependence cancels)
  \[
    T_n
    ~:=~
    t_n(U_n)
    ~=~
    \frac{
      \frac{b_n^\top}{\|b_n\|} U_n
    }{
      \frac{1}{\sqrt{k-1}}
      \left\|
        \bm P_n U_n
      \right\|
    }
    ~=~
    \frac{
      \frac{b_n^\top}{\|b_n\|} \frac{\bm A_n Y^n}{\sigma}
    }{
      \frac{1}{\sqrt{k-1}}
      \left\|
        \bm P_n \frac{\bm A_n Y^n}{\sigma}
      \right\|
    }
    ~\sim~
    \StudT(k-1, \delta \|b_n\|)
  \]
  has non-central Student-$t$ distribution with $k-1$ degrees of freedom and non-centrality parameter $\delta \|b_n\|$. To complete the story, we argue that indeed $T_n$ is a sufficient statistic. Following calculations by \citet{BhowmikK07}, we find that if we define $\nu = k-1$, and $\lambda = \delta \|b_n\|$, the likelihood of $u_n$ factors as required by \eqref{eq:sufficiency} into
  \[
    f^{{\ocvar}_n}_{\delta}(\ocval_n|X^n, \bm{\covvar}_n)
    ~=~
    q_n(\ocval_n) \cdot \dens{f}{\StudT(\nu,\lambda)}{}(\susval_n(\ocval_n))
  \]
  where the $\delta$-independent density factor is given by
  \[
    q_n(\ocval_n) = \frac{1}{2} \pi^{- \nu/2} \sqrt{\nu} \Gamma(\nu/2)
    \left(1+ \frac{t_n(\ocval_n)^2}{\nu}\right)^{\frac{\nu+1}{2}}
  \]
  and $ \dens{f}{\StudT(\nu, \lambda)}{}$ is the density of a noncentral Student-$t$ random variable with $\nu$ degrees of freedom and noncentrality parameter $\lambda$.
The corresponding likelihood ratio
\begin{equation}\label{eq:linreg}
  p^{U_n}_{\delta_+}(u_n | X^n, \bm{\covvar}_n)
  ~=~
  \frac{
    \dens{f}{\StudT(k-1, \delta_+ \|b_n\|)}{}(\susval_n(\ocval_n))
  }{
    \dens{f}{\StudT(k-1, \delta_0 \|b_n\|)}{}(\susval_n(\ocval_n))
  }
\end{equation}
then forms a test martingale under $\nullpoint$ and  \cite{perez2024estatistics} show that it is of the standard group-invariant form considered in that paper, implying that for each $n$ it provides the growth-optimal and relative growth-optimal e-variable for testing $\nullpoint$ vs $\cH_1$.

As we have identified a 1-dimensional sufficient statistic $\susval_n(\ocval_n)$ with distributions indexed by $\delta$ that satisfy the MLR property (as we already noted in the t-test example), by Theorem~\ref{thm:main}, the process \eqref{eq:linreg} is a test supermartingale for all $\delta < \delta_0$, i.e.\ for the larger null hypothesis $\nulloneside$.

\subsection{Label Agnostic Bernoulli}
The examples so far were based on continuous translation and/or scale groups. Here we consider the finite group of label flips. More precisely, we consider i.i.d.\ Bernoulli$(\theta)$ data $Y_1, Y_2, \ldots$ where we contrast the null $\cH_0 : \theta \in \{\theta_0, 1-\theta_0\}$ with the alternative $\cH_1 : \theta \in \{\theta_+, 1-\theta_+\}$ for fixed Bernoulli parameters $\theta_0$ and $\theta_+$ that we assume w.l.o.g.\ lie above $\frac{1}{2}$. To gain invariance under label flips, we coarsen the data to $U_i = \mathbf 1_{Y_i = Y_1}$. With that notation, a sufficient statistic is the larger, say, of the two outcome counts (in either $Y^n$ or $U^n$), i.e.\ $T_n = \max \{\sum_{i=1}^n U_i, n - \sum_{i=1}^n U_i \}$. In those terms, the joint probability mass function is
\[
  f_\theta(u^n)
  ~=~
  \theta^{{T_n}}
  (1-\theta)^{n-{T_n}}
  +
  (1-\theta)^{{T_n}}
  \theta^{n-{T_n}}
\]
which, being a sum of two i.i.d.\ likelihoods, is not itself i.i.d.
To show that the likelihood ratio
\begin{equation}\label{eq:agnostic.bernoulli}
  p_{\theta_+}(u^n)
  ~=~
  \frac{
    f_{\theta_+}(u^n)
  }{
    f_{\theta_0}(u^n)
  }
  ~=~
  \frac{
  \theta_+^{{T_n}}
  (1-\theta_+)^{n-{T_n}}
  +
  (1-\theta_+)^{{T_n}}
  \theta_+^{n-{T_n}}
}{
  \theta_0^{{T_n}}
  (1-\theta_0)^{n-{T_n}}
  +
  (1-\theta_0)^{{T_n}}
  \theta_0^{n-{T_n}}
  }
\end{equation}
is increasing in the sufficient statistic ${T_n} \ge n/2$, it suffices to show positivity of
\[
  \frac{\partial}{\partial {T_n}}  \ln p_{\theta_+}(u^n)
  ~=~
  \frac{\partial}{\partial {T_n}}
  \int_{\theta_0}^{\theta_+}
  \frac{\partial}{\partial \theta}
  \ln f_{\theta}(u^n)
  d \theta
  ~=~
  \int_{\theta_0}^{\theta_+}
  \frac{\partial^2}{\partial \theta \partial {T_n}}
  \ln f_{\theta}(u^n)
  d \theta
  .
\]
We will do so by showing positivity of the integrand. We have, abbreviating $\phi := \frac{\theta}{1-\theta} \ge 1$,
\[
  \frac{\partial^2}{\partial \theta \partial {T_n}} \ln f_{\theta}(u^n)
  ~=~
  \frac{
    (\phi +1)^2
  }{
    \phi  \left(\phi ^{2 T_n-n} + 1\right)^2
  }
  \left(
    2 (2 T_n-n) \phi^{2 T_n-n} \ln \phi
    + \phi ^{2 (2 T_n - n)}
    - 1
  \right)
\]
which is clearly positive, as $2 T_n \ge n$ and $\phi \ge 1$. All in all, for all $\theta_0, \theta_+ \in [1/2,1]$ with $\theta_0 \le \theta_+$ we have, by Theorem~\ref{thm:main}, that \eqref{eq:agnostic.bernoulli} is a supermartingale under all $\theta \in [1/2,\theta_0]$. Symmetrically, for $\theta_+ \le \theta_0$, \eqref{eq:agnostic.bernoulli} is also a supermartingale under $\theta \in [\theta_0, 1]$.

\section{No Free For All}
The examples in \ref{sec:examples} might suggest that, more generally, we obtain a test supermartingale whenever we take expectations under a `more extreme' distribution than the one appearing in the numerator of the likelihood ratio, e.g.\ with the same mean but lighter tails. The following example shows that this is not the case --- apparently, the  precise functional form of the densities under which we take expectations is essential for our argument.

Let $M_n^\delta$ denote the t-test likelihood ratio process defined in \eqref{eq:t-test} with $\delta_+ = \delta$ and $\delta_0 = 0$. We proved that for $\delta > 0$ this process is a supermartingale under data that are i.i.d.\ $N(\mu,\sigma^2)$ for any $\mu \le 0$ and any $\sigma^2 > 0$. A popular generalisation of Gaussians is the sub-Gaussian class. Hence one may be hopeful that the process $M_n^\delta$ is a supermartingale under any i.i.d.\ sub-Gaussian true distribution with mean zero or below. However, this is \emph{not true}, unfortunately. Given that failure, one may perhaps hope it is an e-process. But even that fails. Here we show that it is not even an e-variable for fixed $n$. To do so we draw $Y_i$ i.i.d.\ Rademacher $\pm 1$. As such, all $Y_i$ are zero-mean sub-Gaussian, even conditionally. However, we claim that for all $n \ge 2$, there is $\delta > 0$ small enough such that
\[
  \exo_{Y^n \stackrel{\text{\tiny i.i.d.}}{\sim} \text{Rad}}[M_n^\delta] ~>~ 1
  .
\]
A Taylor series expansion in/around $\delta=0$ to fourth order (which is the first contributing order, also all odd orders are zero) reveals that
\[
  \exo_{Y^n \stackrel{\text{\tiny i.i.d.}}{\sim} \text{Rad}}[M_n^\delta]
  ~=~
  1 + \frac{n-1}{6} \delta^4 + O(\delta^6)
\]
and so indeed we are in trouble: for $\delta$ close enough to zero the second strictly positive term dominates the remainder, lifting the above $>1$.

\end{document}